\documentclass[reqno]{amsart}
\usepackage{amssymb}
\usepackage[mathscr]{eucal}
\usepackage{hyperref}
\usepackage{graphicx}
\usepackage{xcolor}
\usepackage{tikz}
\usetikzlibrary{cd,babel}

\usepackage[font=small,labelformat=empty]{caption}
\usepackage{float}

\theoremstyle{plain}
\newtheorem{prop}{Proposition}[section]
\newtheorem{thm}[prop]{Theorem}

\newtheorem{lem}[prop]{Lemma}

\theoremstyle{definition}
\newtheorem{dfn}[prop]{Definition}
\newtheorem{examples}[prop]{Examples}
\newtheorem{lab}[prop]{}

\newcommand{\isoto}{\overset{\sim}{\to}}

\newcommand{\A}{{\mathbb{A}}}
\newcommand{\N}{{\mathbb{N}}}
\newcommand{\R}{{\mathbb{R}}}

\newcommand{\m}{{\mathfrak{m}}}
\newcommand{\scrO}{{\mathscr{O}}}
\newcommand{\scrV}{{\mathscr{V}}}
\newcommand{\sfS}{\mathsf{S}}

\DeclareMathOperator{\Hom}{Hom}
\DeclareMathOperator{\rk}{rk}
\DeclareMathOperator{\sosx}{sosx}
\DeclareMathOperator{\sxdeg}{sxdeg}
\DeclareMathOperator{\Spec}{Spec}

\newcommand{\id}{\mathrm{id}}
\newcommand{\comp}{\mathbin{\scriptstyle\circ}}
\newcommand{\du}{{\scriptscriptstyle\vee}}
\newcommand{\ol}{\overline}
\newcommand{\plus}{{\scriptscriptstyle+}}

\newcommand{\ex}{\exists\,}

\renewcommand{\subset}{\subseteq}
\renewcommand{\supset}{\supseteq}
\newcommand{\idl}[1]{\langle #1\rangle}

\newcommand{\lpow}{[\![}
\newcommand{\rpow}{]\!]}

\newcommand{\bil}[2]{\langle{#1},{#2}\rangle}

\newcommand{\sa}{semialgebraic}
\newcommand{\Label}[1]{\label{#1}}

\begin{document}

\title
[Smooth hyperbolicity cones are second-order cone representable]
{Smooth hyperbolicity cones\\are second-order cone representable}

\author{Claus Scheiderer}
\address
  {Fachbereich Mathematik und Statistik \\
  Universit\"at Konstanz \\
  78457 Konstanz \\
  Germany}
\email
  {claus.scheiderer@uni-konstanz.de}

\begin{abstract}
Netzer and Sanyal proved that every smooth hyperbolicity cone is a
spectrahedral shadow. We generalize and sharpen this result at the
same time, by showing that every Nash-smooth hyperbolicity cone is
even second-order cone representable (socr). The result is proved as
a consequence of our second theorem, according to which every compact
convex \sa\ set with Nash-smooth boundary of strict positive
curvature is socr. The proof uses the technique of tensor evaluation.
\end{abstract}

\maketitle

\section{Introduction}

A homogeneous real polynomial $f\in\R[x]=\R[x_1,\dots,x_n]$ is said
to be \emph{hyperbolic} with respect to a point $e\in\R^n$ if
$f(e)\ne0$ and if, for every $u\in\R^n$, all complex zeros of the
univariate polynomial $f(te-u)$ are real. Geometrically this says
that every real line through the point $[e]$ in projective
$(n-1)$-space intersects the hypersurface $f=0$ in real points only.
The notion of hyperbolic forms originates in the study of partial
differential operators and goes back to Petrovskii and G\aa rding. In
particular, G\aa rding \cite{ga} proved that the set
$C_e(f)=\{u\in\R^n\colon f(te-u)\ne0$ for $t<0\}$ is a closed convex
cone, called the \emph{hyperbolicity cone} of $f$ (with respect
to~$e$), and that this cone agrees with the closure of the connected
component of $\{u\colon f(u)\ne0\}$ that contains~$e$. Since the
early 2000s, hyperbolic forms have moved into the focus of real
algebraic geometry, algebraic combinatorics and optimization.
According to G\"uler \cite{gu}, the function $-\log|f(x)|$ is a
self-concordant barrier function for the cone $C_e(f)$, a fact that
makes hyperbolicity cones available for convex programming.

If $M$ is a real symmetric matrix, the notation $M\succ0$ means
that $M$ is positive definite; likewise, $M\succeq0$ means that
$M$ is positive semidefinite. Let $A_1,\dots,A_n$ be real symmetric
$d\times d$-matrices and write $A(x)=\sum_{i=1}^nx_iA_i$ for
$x\in\R^n$. Convex cones of the form
$\{x\in\R^n\colon A(x)\succeq0\}$ are called \emph{spectrahedral}. If
there is $e\in\R^n$ with $A(e)\succ0$, the form $f(x)=\det A(x)$ will
be hyperbolic with respect to $e$, with associated hyperbolicity cone
$C_e(f)=\{x\in\R^n\colon A(x)\succeq0\}$. For $n=3$, every hyperbolic
form arises in this way (up to sign), according to the celebrated
\emph{Lax conjecture} that was proved by Helton and Vinnikov (cf.\
\cite{hv,lpr}). For $n\ge4$ variables it is easy to see that there
exist much more hyperbolic forms than just the forms $\det A(x)$ with
$A(e)\succ0$. But it is a major open question whether these give rise
to any new hyperbolicity cones. Rather, according to the so-called
\emph{generalized Lax conjecture (GLC)}, all hyperbolicity cones
should be spectrahedral. In many special cases this has been shown to
be true, but in general the conjecture is wide open, see for instance
\cite{v,NP}.

A weakening of GLC asks whether every hyperbolicity cone is at least
a spectrahedral shadow, i.e., a linear image of a spectrahedral cone.
Even this weaker question is completely open in general. However,
Netzer and Sanyal \cite{ns} proved that the answer is positive when
the hyperbolicity cone $C_e(f)$ in question has smooth boundary. In
other words, they showed that, under the smoothness assumption,
$C_e(f)$ is the solution set of a \emph{lifted linear matrix
inequality} (\emph{lifted LMI}), which means that there exist
symmetric matrices $A_i$ ($i=1,\dots,n$) and $B_j$ ($j=1,\dots,m$,
for some $m$) with
\begin{equation}\Label{cefshadow}%
C_e(f)\>=\>\Bigl\{x\in\R^n\colon\ex y\in\R^m\ \sum_{i=1}^nx_iA_i
+\sum_{j=1}^my_jB_j\succeq0\Bigr\}.
\end{equation}
This result implies that convex optimization over smooth
hyperbolicity cones can in fact be performed as semidefinite programs
(SDP).

Our first main result is:

\begin{thm}\Label{main2}%
Let the form $f\in\R[x]$ be hyperbolic with respect to $e\in\R^n$. If
the hyperbolicity cone $C_e(f)$ has Nash-smooth boundary, there
exists a lifted LMI representation \eqref{cefshadow} of $C_e(f)$ in
which the matrices $A_i$ and $B_j$ are of common block-diagonal
shape, with each diagonal block of size at most $2\times2$.
\end{thm}

The notion of Nash-smooth boundary used here is more general than the
concept of smooth boundary, as used in \cite{ns,hn09,hn10} for
example. Roughly, it only requires the cone $C_e(f)$ to be locally
defined by a smooth Nash inequality, instead of a smooth polynomial
inequality. See Section~\ref{smobopo} for precise definitions. Apart
from this, the main difference to the Netzer-Sanyal theorem \cite{ns}
is that a lifted LMI representation \eqref{cefshadow} can be found
that consists of a collection of $2\times2$ matrices. The cone
$C_e(f)$ is therefore even second-order cone (soc) representable,
meaning that convex optimization over such cones can be performed as
second-order cone programs (SOCP) \cite{BTN}.
The latter are known to perform much faster and much more efficient
than general semidefinite programs.

Our second result sharpens and generalizes a theorem of Helton and
Nie:

\begin{thm}\Label{main1}%
Let $K$ be a compact convex \sa\ set. Assume that the boundary of $K$
is Nash-smooth and has strict positive curvature everywhere. Then $K$
has a lifted LMI representation
\begin{equation}\Label{kshadow}%
K\>=\>\Bigl\{x\in\R^n\colon\ex y\in\R^m\ A_0+\sum_{i=1}^nx_iA_i
+\sum_{j=1}^my_jB_j\succeq0\Bigr\}
\end{equation}
in which the matrices $A_i$ and $B_j$ have a common block-diagonal
structure with all diagonal blocks of size at most $2\times2$.
\end{thm}

See \ref{strquasiconc} and \ref{dfnashsmoth} for the notion of strict
positive curvature. Helton and Nie \cite{hn09,hn10} had proved that
compact convex \sa\ sets are spectrahedral shadows, provided their
boundary is smooth of strict positive curvature (see also
\cite[Thm.~8.6.29]{Sch}). Theorem \ref{main1} covers the larger class
of sets where the boundary is just Nash-smooth. More importantly, it
shows that such sets are even soc-representable.

Fawzi \cite{fa} was the first to observe that the family of
soc-representable sets is much smaller than the family of all
spectrahedral shadows. The gap between the two was studied in detail
by Averkov \cite{av}.
His invariant $\sxdeg(K)$ defines an infinite hierarchy of geometric
complexity for convex sets $K$,
which roughly corresponds to the maximal block size in a
block-diagonal lifted LMI representation \eqref{kshadow} of~$K$.
That $K$ is a spectrahedral shadow means $\sxdeg(K)<\infty$, while
being soc-representable means $\sxdeg(K)\le2$ (see
\ref{spectrahsxdegetc} below for precise details). Our results are
therefore saying that the cones $C_e(f)$ and $K$ in Theorems
\ref{main2} and \ref{main1} have the lowest level of geometric
complexity in this hierarchy, apart from polyhedra.

The paper is structured as follows.
After explaining the necessary technical background, we'll start by
proving Theorem \ref{main1} in Section~\ref{proof1}. The method of
proof is independent of the Helton-Nie approach. Rather it relies on
the technique of tensor evaluation, as introduced in
\cite{sch:sxdeg}. Much like Netzer-Sanyal used the results of Helton
and Nie to prove their main theorem \cite{ns}, we will then use
Theorem \ref{main1} to prove Theorem \ref{main2}
(Section~\ref{proof2}). The proof gets more complicated, though, than
in \cite{ns}, due to the weaker assumptions.

\emph{Some general conventions}:
By an affine algebraic variety over the field $k$ we mean an affine
$k$-scheme $X$ of finite type. The coordinate ring of $X$ is denoted
$k[X]$. Affine $n$-space over $k$ is denoted
$\A^n=\Spec k[x_1,\dots,x_n]$. If $E/k$ is a field extension, the set
of $E$-rational points of $X$ is $X(E)=\Hom_k(k[X],E)$. The ideal of
a ring $A$ generated by $a_1,\dots,a_r\in A$ is written
$\idl{a_1,\dots,a_r}$. We write $\R_\plus:=\{a\in\R\colon a\ge0\}$.


\section{Smooth boundary points of convex sets}\Label{smobopo}%

\begin{lab}\Label{hessian}%
Let $f\colon U\to\R$ be a smooth function defined on an open set $U$
in $\R^n$. Gradient and Hessian of $f$ at $u\in U$ are denoted
$\nabla_f(u)=\bigl(\partial_1f(u),\dots,\partial_nf(u)\bigr)^t$
and $\nabla^2_f(u)=\bigl(\partial_i\partial_jf(u)\bigr)
_{1\le i,j\le n}$, respectively, where
$\partial_i=\frac\partial{\partial x_i}$.

Let $V$ be an affine $\R$-variety, let $u\in V(\R)$ be a
nonsingular point of $V$ and let $f\in\R[V]$ have vanishing order
$\ge2$ at~$u$. The Hessian
of $f$ at $u$ is a well-defined symmetric bilinear form on the
tangent space of $V$ at $u$. If $x_1,\dots,x_n\in\R[V]$ is a regular
system of parameters at $u$, it
is represented by the symmetric matrix
$\bigl(\frac{\partial^2f(u)}{\partial x_i\partial x_j}\bigr)_{i,j}$
of second partial derivatives.
\end{lab}

\begin{lab}\Label{strquasiconc}%
Following \cite{hn10}
we say that a $C^2$-function $f\colon U\to\R$, defined on an open
set $U\subset\R^n$, is \emph{strictly quasi-concave} at $v\in U$ if
$w^t\cdot\nabla^2_f(v)\cdot w<0$ for every (column) vector $w\ne0$ in
$\R^n$ with $\bil w{\nabla_f(v)}=0$.

Let $K$ be a closed convex \sa\ set in $\R^n$ with non-empty
interior, and write $\R[x]=\R[x_1,\dots,x_n]$. Recall
\cite{hn09,hn10,ns}
that $v\in\partial K$ is a \emph{smooth boundary point} of $K$ if
there is a polynomial $h\in\R[x]$ with $\nabla_h(v)\ne0$ such that
$K\cap U=\{\xi\in U\colon h(\xi)\ge0\}$ for some neighborhood $U$ of
$v$. We'll refer to such $h$ as a \emph{positive inequality} for $K$
at~$v$. Given such $v$ and $h$, the boundary of $K$ is said to have
\emph{strict
positive curvature} at $v$ if $h$ is strictly quasi-concave at~$v$,
cf.\ \cite[3.1]{hn09}.

In other words, $\partial K$ has strict positive curvature at $v$ if
and only if the second fundamental form of $\partial K$ at $v$ is
negative definite. Note that the set of smooth boundary points of $K$
is relatively open in $\partial K$. The same is true for the set of
(smooth) boundary points at which $\partial K$ has strict positive
curvature.
\end{lab}

\begin{lab}\Label{poscurvequiv}%
For the proofs of our main results, the following equivalent
characterization will be useful. Let $h\in\R[x]$ be a positive
inequality for $K$ at the smooth boundary point $v$. The linear
polynomial
$$T_v\>=\>\sum_{i=1}^n\frac{\partial h}{\partial x_i}(v)\cdot
(x_i-v_i)$$
will be called the \emph{positive tangent} to $K$ at~$v$. Up to
scaling by a positive real number, $T_v$ depends only on $K$ and~$v$.
By definition, the hyperplane in $\R^n$ defined by $T_v$ is
the unique supporting hyperplane that touches $K$ at~$v$. One checks
easily that $T_v$ is indeed non-negative on $K$.

Let $Y=\scrV(h)$, the hypersurface in $\A^n$ defined by the equation
$h=0$. Then $v$ is a nonsingular point of $Y$. Let $\m_v$ be the
maximal ideal of the local ring $\scrO_{Y,v}$, and let
$\Theta_v\in\R[Y]$ be the restriction of $T_v$ to $Y$. So $\Theta_v$
is the coset of $T_v$ in $\R[x]/\idl h$. Then $\Theta_v$ vanishes of
order $\ge2$ at~$v$.
The class of $\Theta_v$ in $\m_v^2/\m_v^3$ is naturally identified
with the Hessian of $\Theta_v$ at~$u$, see \ref{hessian}. It is
positive definite if and only if $h$ is strictly quasi-convex at $v$
(if and only if $K$ has strict positive curvature at~$v$), see
\ref{strquasiconc}.
\end{lab}

To state our main results in greater generality, we introduce the
following notion that generalizes the concept of smooth boundary
points. Recall that a function $h\colon U\to\R$, defined on an open
set $U\subset\R^n$, is said to be Nash if its graph is semialgebraic
and $f$ is $C^\infty$ (equivalently, analytic). For general
background on Nash functions we refer to \cite[ch.~8]{BCR}.

\begin{dfn}\Label{dfnashsmoth}%
Let $K\subset\R^n$ be a closed convex \sa\ set with non-empty
interior. We say that $v\in\partial K$ is a \emph{Nash-smooth
boundary point} of $K$ if there exists a Nash function $h$, defined
on an open neighborhood $U$ of $v$ in $\R^n$, such that
$\nabla_h(v)\ne0$ and $K\cap U=\{\xi\in U\colon h(\xi)\ge0\}$. Again,
such $h$ will be called a positive (Nash) inequality for $K$ at~$v$.
The boundary $\partial K$ has \emph{strict positive curvature} at $v$
if $h$ is strictly quasi-concave at~$v$. And as before, we define the
positive tangent to $K$ at $v$ by
$$T_v\>=\>\sum_{i=1}^n\frac{\partial h}{\partial x_i}(v)\cdot
(x_i-v_i)$$
(well-defined up to a positive scalar factor).
\end{dfn}

\begin{examples}\Label{nashnotsmo}%
The difference between Definition \ref{dfnashsmoth} and the usual
definition \ref{strquasiconc} is, of course, that local inequalities
for $K$ can be given by local Nash functions in \ref{dfnashsmoth},
and need not be polynomials. Here are a few examples illustrating
this point:
\smallskip

1.\enspace
Consider the polynomial $p=x(x^2+y^2)-(x^4+x^2y^2+y^4)$ and the
convex set $K=\{(x,y)\in\R^2\colon p(x,y)\ge0\}$ in the plane.
The origin $O=(0,0)$ is a boundary point of $K$ and is a singular
point of the irreducible curve $p=0$. Locally at $O$, therefore,
$K$ cannot be defined by a polynomial inequality $h\ge0$ with
$\nabla_h(O)\ne0$. But $p$ is a product
$p=h_1h_2=(x-y^2-y^6\cdots)(x^2+y^2-x^3-x^2y^2-xy^4-y^6\cdots)$ of
two local Nash functions at $O$, the first of which is a local
Nash inequality for $K$.
Note that $\partial K$ has strict positive curvature at~$O$, but not
at the (smooth) boundary point $(0,1)$ (see Fig.~1).
\smallskip

2.\enspace
Another type of example is given by the irreducible polynomial
$p=x^5-x^3y-x^2y-y^3+y^2$ and the convex set
$K=\{(x,y)\colon p(x,y)\ge0$, $y\ge x^3\}$. Other than in the
previous example, the curve $p=0$ has two real branches at
$O=(0,0)$, corresponding to the product decomposition
$p=(y-x^2-x^4-x^5\cdots)(y-y^2-x^3-x^2y+x^5-x^4y\cdots)$. The first
factor is a local Nash inequality for $K$, showing that $\partial K$
has strict positive curvature at~$O$ (Fig.~2).
\smallskip

3.\enspace
The linear symmetric matrix polynomial
$$A(x,y)\>=\>\begin{pmatrix}3+x-2y&-x&x&0\\-x&2-y&0&x\\x&0&y&0\\
0&x&0&y\end{pmatrix}$$
has the determinant $p=x^4-x^3y+2x^2y^2-xy^3+2y^4-5x^2y+2xy^2-7y^3
+6y^2$. The set $S=\{(x,y)\in\R^2\colon A(x,y)\succeq0\}$ is a
compact spectrahedron (see \ref{spectrahsxdegetc}) with Nash-smooth
boundary. At the origin $O=(0,0)$, the boundary $\partial S$ is not
smooth. Rather $p$ factorizes as a product
$p=(3y-x^2+xy+y^2\cdots)(2y-x^2-3y^2\cdots)$ of two Nash functions
that are defined in a neighborhood of~$O$ (Fig.~3):

\begin{figure}[H]
\centering
\captionbox{Fig.~1}[.3\linewidth]{\includegraphics[scale=0.2]{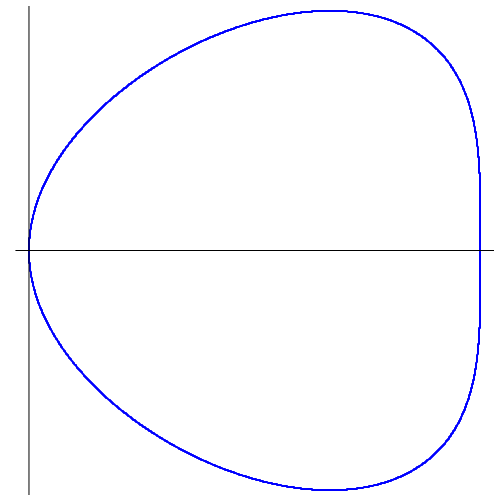}}%
\hfill
\captionbox{Fig.~2}[.3\linewidth]{\includegraphics[scale=0.2]{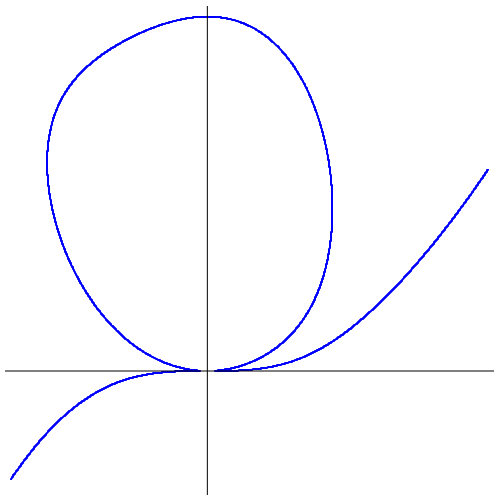}}%
\hfill
\captionbox{Fig.~3}[.3\linewidth]{\includegraphics[scale=0.2]{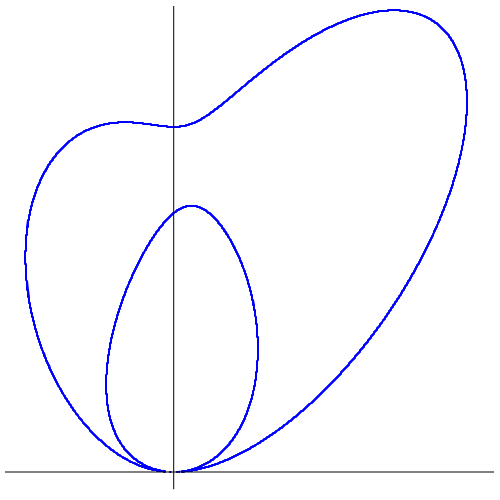}}%
\end{figure}
\end{examples}

\begin{lab}\Label{nashetalpoly}%
Let $v$ be a Nash-smooth boundary point of the closed convex set
$K\subset\R^n$, and let the Nash function $h$, defined on some
neighborhood of $v$, be a positive inequality for $K$ at $v$
(Definition \ref{dfnashsmoth}). We generalize characterization
\ref{poscurvequiv} to the Nash-smooth case, as follows. Using the
Artin-Mazur description of Nash functions (\cite{am}, see also
\cite[Thm.~8.1.4]{BCR}) we may lift the local Nash description
$h\ge0$ of $K$ to a polynomial description on some \'etale cover.
More specifically, there exists an \'etale morphism
$\pi\colon X\to\A^n$ (with $X$ a smooth affine $\R$-variety) and a
point $u\in X(\R)$ with $\pi(u)=v$, plus a polynomial $g\in\R[X]$,
such that $g(\xi)=h(\pi(\xi))$ for every $\xi$ in some neighborhood
of $u$ in $X(\R)$. In other words, $h=g\comp s$ locally around $v$,
where $s$ is the unique local $C^\infty$-section of
$\pi\colon X(\R)\to\R^n$ with $s(v)=u$.
Note that $s$ is a local diffeomorphism between a neighborhood of $v$
in $\R^n$ and a neighborhood of $u$ in $X(\R)$.

Pulling back the sequence $x_1-v_1,\dots,x_n-v_n$ from $\A^n$ to $X$
via $\pi$, we get a regular sequence of parameters $y_1,\dots,y_n$ on
$X$ at $u$. The positive tangent $T_v\in\R[x]$ (to $K$ at $v$, see
\ref{dfnashsmoth}) gets pulled back to the element
\begin{equation}\Label{pullbackT2t}%
t_u\>:=\>\sum_{i=1}^n\frac{\partial g}{\partial y_i}(u)\cdot y_i
\end{equation}
in $\R[X]$.
Let $Y=\scrV_X(g)$ be the subvariety $g=0$ of $X$, and let $\tau_u$
be the image of $t_u$ in $\R[Y]=\R[X]/\idl g$. Since $\pi$ gives a
local diffeomorphism (from a neighborhood of $u$ in $X(\R)$ to a
neighborhood of $v$ in $\R^n$), $\tau_u$ vanishes on $Y$ of order
$\ge2$ at $u$. Moreover, $\partial K$ has strict positive curvature
at $v$ if and only if the Hessian of $\tau_u$ at $u$ is positive
definite (see \ref{hessian}).
\end{lab}


\section{Proof of Theorem \ref{main1}}\Label{proof1}%

\begin{lab}\Label{spectrahsxdegetc}%
A set $K\subset\R^n$ is a \emph{spectrahedron} if $K$ is the solution
set of a \emph{linear matrix inequality (LMI)}
$A_0+\sum_{i=1}^nx_iA_i\succeq0$, where the $A_i$ are real symmetric
matrices of some size.
Linear images of spectrahedra are called \emph{spectrahedral
shadows}.

Let $K$ be a convex \sa\ set in $\R^n$. The \emph{semidefinite
extension degree} of $K$, denoted $\sxdeg(K)$, was defined by Averkov
\cite{av} as the smallest integer $d$ such that $K$ is a linear image
of a finite intersection of spectrahedra, each of which is described
by an LMI of size at most $d\times d$. So $\sxdeg(K)\le d$ means that
$K$ can be represented in the form
$$K\>=\>\Bigl\{x\in\R^n\colon\ex y\in\R^m\ M_1(x,y)\succeq0,\dots,
M_r(x,y)\succeq0\Bigr\}$$
for some $m,\,r\ge0$ where, for each $\nu=1,\dots,r$,
$M_\nu(x,y)=A_{\nu0}+\sum_{i=1}^nx_iA_{\nu i}+\sum_{j=1}^my_j
B_{\nu j}$ is a linear reak symmetric matrix polynomial of size at
most $d\times d$. The hierarchy defined by this invariant is strict,
for example since the cone $\sfS^d_\plus$ of psd real symmetric
$d\times d$-matrices has $\sxdeg(\sfS^d_\plus)=d$ for every $d\ge1$
\cite{av}. By definition, $\sxdeg(K)\le1$ holds if and only if $K$ is
a polyhedron, while $\sxdeg(K)\le2$ if and only if $K$ is
second-order cone representable \cite[Example 1.5]{sch:sxdeg}. Note
that $\sxdeg(K)<\infty$ if and only if $K$ is a spectrahedral shadow.
\end{lab}

\begin{lab}\Label{tenseval}%
In this terminology, Theorem \ref{main1} asserts that
$\sxdeg(K)\le2$ holds for $K$ as in the theorem. A useful technique
for working with the invariant $\sxdeg(K)$ is the concept of tensor
evaluation, introduced in \cite{sch:sxdeg}. Let $X$ be an affine
$\R$-variety and let
$R\supset\R$ be a real closed field. Write $R[X]=\R[X]\otimes R$
(tensor product over~$\R$) and let $f\in R[X]$. If $\xi\in X(R)$ is
an $R$-rational point of $X$ (i.e., a homomorphism
$\xi\colon\R[X]\to R$), the \emph{tensor evaluation}
$f^\otimes(\xi)$ (of $f$ at~$\xi$) is the image of $f$ under the
ring homomorphism
$$R[X]\>=\>\R[X]\otimes R\xrightarrow{\xi\otimes\id}R\otimes R\>=\>
R\otimes_\R R.$$
Note that the usual evaluation $f(\xi)$ (the image of $f$ under the
homomorphism $\xi$) is the result of applying the product map
$R\otimes R\to R$ to $f^\otimes(\xi)$.

Given a \sa\ subset $S$ of $X(\R)$, let $S_R$ denote the base field
extension of $S$ to $R$ \cite[Sect.~4.1]{Sch}. This is the subset of
$X(R)$ that is described by the same finite system of polynomial
inequalities as~$S$. Given $\theta\in R\otimes R$, the tensor rank
$\rk(\theta)$ of $\theta$ is the minimal number $r\ge0$ such that
$\theta=\sum_{i=1}^ra_i\otimes b_i$ for suitable $a_i,\,b_i\in R$. If
$\theta$ is a sum of squares in $R\otimes R$, let $\sosx(\theta)$ be
the smallest integer $d\ge0$ such that there is an identity
$\theta=\sum_i\theta_i^2$ in $R\otimes R$ with $\rk(\theta_i)\le d$
for all~$i$.
Put $\sosx(\theta)=\infty$ if $\theta$ is not a sum of squares. To
prove Theorem \ref{main1} we are going to apply the following
criterion. It is a particular case of \cite[Thm.~3.10]{sch:sxdeg}:
\end{lab}

\begin{thm}\Label{critenseval}%
Let $K\subset\R^n$ be a compact and convex \sa\ set, let
$P=\{f\in\R[x_1,\dots,x_n]\colon\deg(f)\le1$, $f|_K\ge0\}$, and let
$E$ denote the set of elements in $P$ that span an extreme ray of the
cone~$P$. Let $d\ge1$ be an integer. Then $\sxdeg(K)\le d$ if and
only if $\sosx f^\otimes(\xi)\le d$ holds for every real closed field
$R\supset\R$, every $f\in E_R$ and every $\xi\in\partial K_R$.
\end{thm}

\begin{lab}\Label{tangenten}%
Let $K\subset\R^n$ be a closed convex \sa\ set, and let $P$ be the
cone of linear polynomials $f$ in $\R[x]=\R[x_1,\dots,x_n]$ with
$f|_K\ge0$. We assume that $K$ has non-empty interior, so the cone
$P$ is closed and pointed. Assume further that the boundary of $K$
is Nash-smooth (\ref{dfnashsmoth}). Given $v\in\partial K$, let
$T_v\in P$ be a positive tangent to $\partial K$ at $v$ (see
\ref{dfnashsmoth}).
Up to scaling by a positive real number, $T_v$ is the only element in
$P$ that vanishes at~$v$.
If the boundary $\partial K$ has strict positive curvature
everywhere, then $T_v(v')>0$ holds for any two different points
$v\ne v'$ in $\partial K$.
\end{lab}

\begin{lem}\Label{xrays}%
Assume that $K$ is compact convex \sa\ with Nash-smooth boundary.
Then the extreme rays of the cone $P$ are the $\R_\plus T_v$, for
$v\in\partial K$. In particular, $P$ is the conic hull of
$\{T_v\colon v\in\partial K\}$.
\end{lem}

\begin{proof}
Let $v\in\partial K$. Then $T_v$ generates an extreme ray of $P$,
since $T_v=p+q$ with $p,\,q\in P$ implies $p(v)=0$ and therefore
$p\in\R_\plus T_v$, see \ref{tangenten}.
Conversely, if $t\in P$ generates an extreme ray of $P$, there is a
point $v\in\partial K$ with $t(v)=0$ since $K$ is compact. Therefore
$t\in\R_\plus T_v$, again by \ref{tangenten}. Every closed and
pointed convex cone is the Minkowski sum of its extreme rays, which
implies the last assertion.
\end{proof}

\begin{lab}\Label{suffshow}%
To prove Theorem \ref{main1}, it therefore suffices (by Theorem
\ref{critenseval}) to show: For any real closed field $R\supset\R$
and any $R$-points $\xi,\,\eta\in\partial K_R$, the tensor evaluation
of every positive tangent $T_\eta$ at $\xi$ satisfies
$\sosx T_\eta^\otimes(\xi)\le2$.

This can be narrowed down further. For clarity we will sometimes
write $K_\R$ in the following, instead of~$K$.
Let $\scrO_R=\{a\in R\colon\ex n\in\N$ with $|a|<n\}$, the canonical
valuation ring of $R$, and let $\scrO_R\to\R$, $a\mapsto\ol a$ denote
the residue map. For $\xi\in\scrO_R^n$ we call $\ol\xi\in\R^n$ the
\emph{specialization} of $\xi$ in $\R^n$. If $K\subset\R^n$ is \sa\
and $\xi\in K_R\cap\scrO_R^n$, then $\ol\xi\in\R^n$ lies in the
closure of $K_\R=K$. In particular, when $K$ is compact, the base
field extension $K_R$ of $K$ is contained in $\scrO_R^n$ and
satisfies $\ol\xi\in K_\R$ for every $\xi\in K_R$.

Now assume that $K$ is compact convex with Nash-smooth boundary of
strict positive curvature. Then it suffices to consider criterion
\ref{critenseval} in the case where the specializations of $\xi$ and
$\eta$ agree:
\end{lab}

\begin{lem}\Label{spezndiff}%
Let $R\supset\R$ be real closed.
If $\xi,\,\eta\in\partial K_R$ are such that $\ol\xi\ne\ol\eta$ in
$\partial K_\R$, then $\sosx T_\eta^\otimes(\xi)=1$.
\end{lem}

\begin{proof}
First note that $(aT_\eta)^\otimes(\xi)=(1\otimes a)\cdot
T_\eta^\otimes(\xi)$ for any $a\in R$. So scaling $T_\eta$ by a
positive number in $R$ doesn't change $\sosx T_\eta^\otimes(\xi)$.
Therefore we may assume that $T_\eta\in R[x]$ has coefficients in
$\scrO_R$ and that $\ol{T_\eta}\ne0$ in $\R[x]$. Then $\ol{T_\eta}$
is a positive tangent to $K_\R$ at $\ol\eta\in\partial K_\R$. Since
$\ol\xi\ne\ol\eta$ are different boundary points of $K_\R$, we have
$\ol{T_{\eta}}({\ol\xi})>0$, as remarked in \ref{tangenten}.
So $T_\eta^\otimes(\xi)$ is a tensor in
$\scrO_R\otimes\scrO_R\subset R\otimes R$ whose residue in
$\R\otimes\R=\R$ is strictly positive. By
\cite[Prop.\ 3.5]{sch:sxdeg} this implies
$\sosx T_\eta^\otimes(\xi)=1$.
\end{proof}

\begin{lab}\Label{etalift}%
According to Lemma \ref{spezndiff}, we only need to study
$T_\eta^\otimes(\xi)$ in the case where $\xi,\,\eta\in\partial K_R$
have the same specialization $v$ in $K_\R$. Let $\pi\colon X\to\A^n$
be an \'etale morphism as in \ref{nashetalpoly}, with $\pi(u)=v$, and
let $g\in\R[X]$ be the lift of a positive Nash inequality for $K$ at
$v$. Moreover write $Y=\scrV_X(g)$ as in \ref{nashetalpoly}. Given
$\xi,\,\eta\in\partial K_R$ with $\ol\xi=\ol\eta=v$, let
$\xi',\,\eta'$ be their lifts to $Y(R)$ with
$\ol{\xi'}=\ol{\eta'}=u$. Finally let $t_{\eta'}\in R[X]$ be the lift
of $T_\eta\in\R[x]$ as in \eqref{pullbackT2t}. Then
\begin{equation}\Label{obsetalift}%
t_{\eta'}^\otimes(\xi')\>=\>T_\eta^\otimes(\xi)
\end{equation}
holds in $R\otimes R$, which means that the tensor evaluation of
$T_\eta$ at $\xi$ can be calculated as tensor evaluation of
$t_{\eta'}$ at $\xi'$ on the \'etale cover $X$.
Indeed, \eqref{obsetalift} follows from the commutative triangle
$$\begin{tikzcd}[column sep=large,row sep=small]
\R[X]\otimes R \arrow{dr}{\xi'\otimes\id} \\
& R\otimes R \\
\R[x]\otimes R \arrow[swap]{ur}{\xi\otimes\id}
  \arrow{uu}{\pi^*\otimes\id}
\end{tikzcd}$$
since $(\pi^*\otimes\id)(T_\eta)=t_{\eta'}$ by the definition of
$t_{\eta'}$ in \ref{nashetalpoly}.

Recall that we need to show that the tensor \eqref{obsetalift} has
sosx-invariant $\le2$, for $\xi,\,\eta\in\partial K_R$ with
$\ol\xi=\ol\eta=v$. Using the observation just made, we'll work with
the tensor $t^\otimes_{\eta'}(\xi')$ on $X$ from now on. In doing
this we'll drop the ``prime'' notation when denoting points in
$X(R)$, so we'll write $\xi$ and $\eta$ for what was $\xi'$ and
$\eta'$ before.
In order to get an expression of the tangents $t_\eta$ that is
uniform in $\eta$, we'll be working on $X\times X$.
\end{lab}

\begin{lab}\Label{tu}%
Let $X$ be a smooth affine $\R$-variety and let $A=\R[X]$. Since our
discussion is essentially local in nature (on $X$), we may assume
that $x_1,\dots,x_n\in A$ are such that the $A$-module
$\Omega_A=\Omega_{A/\R}$ of K\"ahler differentials is freely
generated by $dx_1,\dots,dx_n$. Fix an element $g\in A$ and write
$Y=\scrV_X(g)$. Write $\partial_i=\frac\partial{\partial x_i}$, so
the differential of $g$ in $\Omega_A$ is
$dg=\sum_{i=1}^n(\partial_ig)\,dx_i$. A point $u\in Y(\R)$ is a
nonsingular point of $Y$ if $(\partial_ig)(u)\ne0$ for some
index~$i$. Given such $u$, write
$t_u=\sum_{i=1}^n(\partial_ig)(u)\cdot(x_i-x_i(u))\in A$, cf.\
\ref{poscurvequiv} and \ref{nashetalpoly}.
\end{lab}

\begin{lab}\Label{atensora}%
The ring $A\otimes A=A\otimes_\R A$ will be considered as an
$A$-algebra via the second tensor component. So we write
$p\cdot\theta:=(1\otimes p)\cdot\theta$ for $p\in A$ and
$\theta\in A\otimes A$. Let the map $\delta\colon A\to A\otimes A$ be
given by $\delta(p)=p\otimes1-1\otimes p$, and let $I$ denote the
kernel of the product map $A\otimes A\to A$. The ideal $I$ of
$A\otimes A$ is generated by the elements $\delta(p)$ ($p\in A$). The
$A$-module $\Omega_A$ is naturally isomorphic to
$I/I^2$, via $\Omega_A\isoto I/I^2$, $dp\mapsto\delta(p)+I^2$.
Consider the tensor
\begin{equation}\Label{dfnt}%
t\>:=\>\sum_{i=1}^n(\partial_ig)\cdot\delta(x_i)
\end{equation}
in $I$. The class of $t$ in $I/I^2=\Omega_A$ is $dg$. The tensor $t$
parametrizes the ``tangents'' $t_u$ in the sense that, for
$u\in Y(\R)$, one has $t_u=\varphi_u(t)$ where
$\varphi_u\colon A\otimes A\to A$ is the homomorphism
$p\otimes q\mapsto p\cdot q(u)$.

Write $B=\R[Y]=A/\idl g$, and let $J$ be the kernel of the product
map $B\otimes B\to B$. By $\tau$ we denote the image of $t$ under the
natural epimorphism $A\otimes A\to B\otimes B$. Since
$dg\in\Omega_X\cong I/I^2$ restricts to zero in
$\Omega_Y\cong J/J^2$, the element $\tau$ lies in $J^2$.

The point for our discussion is this: Given $R\supset\R$ real closed,
and given $\xi,\,\eta\in Y(R)$,
tensor evaluation $t_\eta^\otimes(\xi)$ is the image of the tensor
$\tau\in J^2\subset B\otimes B$ under the homomorphism
\begin{equation}\Label{xieta}%
\xi\otimes\eta\colon B\otimes B\to R\otimes R,\quad
p\otimes q\mapsto p(\xi)\otimes q(\eta).
\end{equation}
Now combine this remark with Lemma \ref{spezndiff} and Observation
\ref{etalift}. Since the strict positive curvature assumption for
$\partial K$ implies that the Hessian of $\tau_u$ at $u$ is positive
definite (see \ref{nashetalpoly}), it suffices for the proof of
Theorem \ref{main1} to establish the following proposition:
\end{lab}

\begin{prop}\Label{suff2show}%
Let $X$ be an affine $\R$-variety, put $A=\R[X]$ and assume that
$\Omega_A$ is freely generated by $dx_1,\dots,dx_n$. Let $g\in A$,
put $Y=\scrV_X(g)$ and $B=\R[Y]=A/\idl g$, and let $\tau$ be the
image of
$t=\sum_{i=1}^n(\frac{\partial g}{\partial x_i})\,\delta(x_i)$ in
$B\otimes B$. Let $u\in Y(\R)$ be a smooth point of $Y$, let
$\tau_u\in B$ be the image of $\tau$ under
$\varphi_u\colon B\otimes B\to B$, $p\otimes q\mapsto p\cdot q(u)$,
and assume that the Hessian of $\tau_u$ at the point $u$ is positive
definite. Then, for $R\supset\R$ real closed and any points
$\xi,\,\eta\in Y(R)$ with $\ol\xi=\ol\eta=u$, the image of $\tau$
under $\xi\otimes\eta\colon B\otimes B\to R\otimes R$ satisfies
$\sosx\,(\xi\otimes\eta)(\tau)\le2$.
\end{prop}

The technical key step in the proof of Proposition \ref{suff2show}
will be the next lemma:

\begin{lem}\Label{schweregeburt}%
Assume the hypotheses of Proposition \ref{suff2show}, and let
$\Phi_u\colon B\otimes B\to\R$, $p\otimes q\mapsto p(u)q(u)$ denote
evaluation at~$u$. Then there is an identity
\begin{equation}\Label{ident4tau}%
\tau\>=\>\sum_\nu\gamma_\nu\cdot\delta(z_\nu)^2
\end{equation}
in $B\otimes B$ with $z_\nu\in B$ and $\gamma_\nu\in B\otimes B$,
such that $\Phi_u(\gamma_\nu)>0$ holds for every~$\nu$.
\end{lem}

\begin{proof}
Let $\m_u$ denote the maximal ideal of $B$ at $u$. Since $Y$ is
smooth at $u$, the local ring $B_{\m_u}$ is regular (of dimension
$n-1$). Consider the evaluation map
$\varphi_u\colon B\otimes B\to B$, $p\otimes q\mapsto p\cdot q(u)$
(for $p,\,q\in B$). Then $\varphi_u(J)=\m_u$,
so $\varphi_u$ induces
a map $\ol\varphi_u\colon J^2/J^3\to\m_u^2/\m_u^3$. It corresponds to
sending a coset in $J^2/J^3$ to its value at the point $u$, which
is a symmetric bilinear form on the tangent space $(\m_u/\m_u^2)^\du$
of $Y$ at~$u$. By saying that a tensor $\alpha\in J^2$ is positive
definite at $u$, we mean that
$\ol\varphi_u(\alpha+J^3)\in\m_u^2/\m_u^3$ is positive definite, as a
symmetric bilinear form on $(\m_u/\m_u^2)^\du$.
With this terminology, the hypothesis in \ref{suff2show} on the
Hessian of $\tau_u$ says that the tensor $\tau$ is positive definite
at~$u$.

Let $y_1,\dots,y_r\in B$ be a finite sequence of elements in $\m_u$
such that $J$ is generated by $\delta(y_1),\dots,\delta(y_r)$, as an
ideal in $B\otimes B$.
Let $z_1,\dots,z_m$ (with $m=r^2$) be the sequence consisting of the
elements $y_i$ and $y_i\pm y_j$ (for $1\le i<j\le r$), in some order.
Then the ideal $J^2$ is generated by
$\delta(z_1)^2,\dots,\delta(z_m)^2$.
Fix a real number $\epsilon>0$ for which the tensor
\begin{equation}\Label{taucorr}%
\tau'\>:=\>\tau-\epsilon\sum_{i=1}^m\delta(z_i)^2
\end{equation}
is still positive definite at $u$.
Let $p_{ij}=p_{ji}$ ($1\le i,j\le r$) be elements in $B$ with
\begin{equation}\Label{taudarst}%
\tau'\>\equiv\>\sum_{i,j=1}^rp_{ij}\,\delta(y_i)\delta(y_j)
\quad (\text{mod }J^3),
\end{equation}
and write $a_{ij}=p_{ij}(u)$ for $1\le i,j\le r$. By construction we
have $\varphi_u(\tau')\equiv\sum_{i,j=1}^ra_{ij}y_iy_j$
(mod~$\m_u^3$), as elements in $\m_u^2/\m_u^3$. The real symmetric
$r\times r$-matrix $M=(a_{ij})_{i,j}$ is positive semidefinite.
So there exists an invertible real $r\times r$-matrix $S=(s_{ij})$
such that $S^tMS$ is a diagonal matrix with nonnegative entries. Let
$\tilde y_1,\dots,\tilde y_r\in\m_u$ be defined by
$y_i=\sum_{j=1}^rs_{ij}\tilde y_j$ ($i=1,\dots,r$). This gives
$$\varphi_u(\tau')\>\equiv\>\sum_{j=1}^rb_j\tilde y_j^2\quad
(\text{mod }\m_u^3)$$
with real numbers $b_j\ge0$.
In other words, using this linear coordinate change we may assume
from the beginning that the $p_{ij}$ in \eqref{taudarst} satisfy
$p_{ij}(u)=0$ for $i\ne j$ and $p_{ii}(u)\ge0$ for all~$i$. With
this assumption, rewrite \eqref{taudarst} as
\begin{equation}\Label{eqref2}%
\tau'\>\equiv\>\sum_{i=1}^rp_{ii}\delta(y_i)^2+\frac14
\sum_{i\ne j}p_{ij}\cdot\bigl((\delta(y_i+y_j)^2-\delta(y_i-y_j)^2
\bigr)\quad(\text{mod }J^3).
\end{equation}
By the choice of the $z_i$, the right hand side in \eqref{eqref2} is
a sum $\sum_{i=1}^mp_i\delta(z_i)^2$, where each coefficient
$p_i\in B$ satisfies $p_i(u)\ge0$. Using \eqref{taucorr} we thus
have
\begin{equation}\Label{taurichtig}%
\tau\>=\>\sum_{i=1}^mq_i\delta(z_i)^2+\mu
\end{equation}
with $\mu\in J^3$ and with $q_i\in B$ satisfying
$q_i(u)\ge\epsilon>0$ for all~$i$. Since
$J^3=\sum_{i=1}^mJ\delta(z_i)^2$, we can amalgamate $\mu$ into the
sum and get a representation of $\tau$ as claimed in the proposition.
\end{proof}

\begin{lab}\Label{endofproof}%
Now the proof of Proposition \ref{suff2show} (and therefore of
Theorem \ref{main1}) is readily completed. Let $\xi,\,\eta\in Y(R)$
with $\ol\xi=\ol\eta=u$, and fix an identity \eqref{ident4tau} for
$\tau$ as in \ref{schweregeburt}. Then
$$t_\eta^\otimes(\xi)\>=\>(\xi\otimes\eta)(\tau)\>=\>
\sum_\nu(\xi\otimes\eta)(\gamma_\nu)\cdot\bigl(z_\nu(\xi)\otimes1
-1\otimes z_\nu(\eta)\bigr)^2.$$
For each $\nu$ the real number $\Phi_u(\gamma_\nu)$ is strictly
positive. Therefore the residue of the tensor
$(\xi\otimes\eta)(\gamma_\nu)\in\scrO_R\otimes\scrO_R$ in
$\R\otimes\R=\R$ is strictly positive as well, which implies
$\sosx\,(\xi\otimes\eta)(\gamma_\nu)=1$ by
\cite[Prop.~3.5]{sch:sxdeg}. Now it is obvious that
$\sosx t_\eta^\otimes(\xi)\le2$.
\qed
\end{lab}


\section{Proof of Theorem \ref{main2}}\Label{proof2}%

In this section we are going to derive Theorem \ref{main2} from
Theorem \ref{main1}. First recall some terminology.

\begin{lab}\Label{rzpoly}%
A polynomial $p\in\R[x]=\R[x_1,\dots,x_n]$, not necessarily
homogeneous, is \emph{real-zero} (or an \emph{RZ-polynomial}, or just
\emph{RZ} for short) with respect to $e\in\R^n$ if $p(e)\ne0$ and if,
for every $u\in\R^n$, all complex zeros of the univariate polynomial
$p(te-u)\in\R[t]$ are real. Clearly, this is the inhomogeneous
version of hyperbolic polynomials: If $f\in\R[x]$ is a form that is
hyperbolic with respect to~$e$, the restriction of $f$ to any affine
hyperplane through $e$ is RZ with respect to~$e$. Conversely, if
$p\in\R[x]$ is RZ with respect to~$e$ and $d=\deg(p)$, the
homogenization $p^h(x_0,x)=x_0^d\,p(\frac x{x_0})$ of $p$ will be
hyperbolic with respect to~$(1,e)\in\R^{n+1}$.

If $p$ is real-zero with respect to~$e\in\R^n$, we associate with $p$
the \emph{rigidly convex set}
$$S_e(p)\>:=\>\bigl\{u\in\R^n\colon p((1-t)e+tu)\ne0\text{ for }
0\le t<1\bigr\}.$$
The set $S_e(p)$ is convex and is the closure of the connected
component of $\{x\colon p(x)\ne0\}$ that contains~$e$ (this follows
from the corresponding facts for the hyperbolic form $p^h$ mentioned
in the introduction). The former Lax conjecture can be stated as the
following theorem:
\end{lab}

\begin{thm}\Label{lax}%
\emph{(Helton-Vinnikov \cite{hv,lpr})}
Assume that $p(x_1,x_2)$ is a real-zero polynomial with respect to
$e=(0,0)\in\R^2$ with $p(e)=1$, and let $d=\deg(p)$. Then there exist
real symmetric $d\times d$-matrices $A_1,\,A_2$ such that
$$p(x_1,x_2)\>=\>\det(I+x_1A_1+x_2A_2).$$
\end{thm}

For the proof of Theorem \ref{main2} we'll replace the hyperbolicity
cone $C_e(f)$ by a suitable affine hyperplane section (see
\ref{lastep} below), similar as Netzer and Sanyal did in \cite{ns}.
Then the main technical step still missing is the following
proposition:

\begin{prop}\Label{keystep}%
Let $p\in\R[x]$ be an RZ-polynomial with respect to $e\in\R^n$ and
let $u$ be a Nash-smooth boundary point of $S=S_e(p)$. Assume that
$\partial S$ doesn't contain a nondegenerate line segment passing
through~$u$. Then $\partial S$ has strict positive curvature in~$u$.
\end{prop}

See \cite[Lemma 2.4(ii)]{ns} for a similar statement, which however
is easier to prove due to the ordinary smoothness assumption there.
For a non-smooth example of a situation as in Proposition
\ref{keystep}, see \ref{nashnotsmo}.3 above.
Note that the non-existence of line segments in $\partial S$ is
obviously necessary for the conclusion to hold.

We first prove an easy lemma:

\begin{lem}\Label{isi}%
Let $f\in\R[x]$ be a form that is hyperbolic with respect to
$e\in\R^n$, and let $u$ be a Nash-smooth boundary point of $C_e(f)$.
If $H\subset\R^n$ is any affine-linear subspace containing $e$ and
$u$, the point $u$ is a Nash-smooth boundary point of $H\cap C_e(f)$.
\end{lem}

\begin{proof}
Let $h$ be a positive Nash inequality for $C:=C_e(f)$, defined on a
neighborhood of $u$ in $\R^n$. Consider the Nash function
$g(t):=h(te+(1-t)u)$, defined on some open intervall
$\left]-\epsilon,\epsilon\right[\subset\R$. We claim that
$g'(0)\ne0$. Since the gradient $\nabla_h(u)\ne0$ by hypothesis,
assuming $g'(0)=0$ would imply that $\bil{u-e}{\nabla_h(u)}=0$. Thus
the line connecting $u$ and $e$ would be tangent to $\partial C$ at
$u$, which is absurd.
\end{proof}

\begin{proof}[Proof of Proposition \ref{keystep}]
We may assume $e=0$ and $p(e)=1$. Let $h$ be a local positive Nash
inequality for $S$ at $u$, and fix a vector $0\ne v\in\R^n$ with
$\bil{\nabla_h(u)}v=0$.
We have to show $v^t\nabla^2_h(u)v<0$. Let the polynomial
$g\in\R[x,y]$ be defined by $g(x,y)=p((1-y)u+xv)$. Then $g(x,y)$ is
an RZ-polynomial with respect to $(0,1)$ such that the origin
$O=(0,0)$ lies on the boundary of the rigidly convex set
$K:=S_{(0,1)}(g)$.
The local Nash inequality $h$ for $S$ at $u$ corresponds to a Nash
factor $\tilde g(x,y)$ of $g(x,y)$ in a neighborhood of~$O$, and we
have $v^t\nabla^2_h(u)v=\partial_x^2\tilde g(0,0)$, the second
derivative of $\tilde g(x,0)$ at $x=0$. Therefore we have to show
$\partial_x^2\tilde g(x,0)<0$. Note that the line $y=0$ is tangent to
the local Nash curve $\tilde g(x,y)=0$ at the origin.

It may happen that $g(x,0)\equiv0$ identically. Write $g=y^kG$ with
$k\ge0$ and $G\in\R[x,y]$ such that $G(x,0)\not\equiv0$. Then $G$ is
still an RZ-polynomial with respect to $(0,1)$, and the rigidly
convex set defined by $G$ is the same as for $g$, i.e.\
$S_{(0,1)}(G)=S_{(0,1)}(g)=K$. Moreover, $\tilde g(x,y)$ is a local
Nash factor of $G$ as well, since otherwise $\partial K$ would
contain $[-\gamma,\gamma]\times\{0\}$ for some $\gamma>0$,
contradicting the hypothesis. To complete the proof we may therefore
replace $g$ by $G$. Returning to the previous notation, this means
that we can assume from now on that $g(x,y)$ is not divisible by~$y$.

According to the (ex) Lax conjecture \ref{lax}, there are real
symmetric $d\times d$-matrices $U,\,V$ such that
$g(x,y)=\det(I+xV+(1-y)U)$.
Since $O=(0,0)$ is a boundary point of $K$, we have $I+U\succeq0$ and
$\det(I+U)=0$.
Let $m=\dim\ker(I+U)$, so $m\ge1$. After a suitable orthogonal change
we may assume that
$I+U=\bigl(\begin{smallmatrix}E&0\\0&0\end{smallmatrix}\bigr)$,
written as a block matrix with diagonal blocks of size $d-m$ and $m$,
respectively, and with $E\succ0$. Writing $V$ in the same way as
$V=\bigl(\begin{smallmatrix}A&B^t\\B&C\end{smallmatrix}\bigr)$, we
have $U=\bigl(\begin{smallmatrix}E-I&0\\0&-I\end{smallmatrix}\bigr)$
and
\begin{equation}\Label{gxyblockdet}%
g(x,y)\>=\>\det\begin{pmatrix}E+xA+y(I-E)&xB^t\\xB&xC+yI
\end{pmatrix}.
\end{equation}
The order of vanishing of $g$ at $O$ is $m$, the multiplicity of the
root $y=0$ of $g(0,y)=y^m\cdot\det(yI+(1-y)E)$.

By Weierstrass preparation there exists a (unique) convergent power
series $u\in\R\lpow x,y\rpow$ with $u(0,0)\ne0$ such that
$P(x,y)=u(x,y)g(x,y)$ is a Weierstrass polynomial in
$\R\lpow x\rpow[y]$.
In other words,
$$P(x,y)\>=\>y^m+a_1(x)y^{m-1}+\cdots+a_m(x)$$
with convergent series $a_i\in\R\lpow x\rpow$ that satisfy $a_i(0)=0$.
Since $g$ is an RZ-polynomial, it is easy to see that $P$ factors as
$P(x,y)=\prod_{i=1}^m\bigl(y-g_i(x)\bigr)$ where the $g_i$ are
(convergent) power series (and not just Puiseux series) in
$\R\lpow x\rpow$.
Since $O$ is a Nash-smooth boundary point of $K$ by hypothesis, each
of the local curves $y=g_i(x)$ ($i=1,\dots,m$) has a horizontal
tangent at the origin.
None of the $g_i(x)$ vanishes identically,
so we have $g_i(x)=c_ix^{k_i}+{}$(higher order terms) with $c_i\ne0$
and $k_i\ge2$ for each~$i$. This means that the (lowest) degree~$m$
subform of $g(x,y)$ is $cy^m$ with $c\ne0$. By \eqref{gxyblockdet},
on the other hand, this subform is
$$\det(E)\cdot\det(xC+yI)\>=\>\det(E)\cdot x^m\cdot
p_{-C}(\frac yx)$$
where $p_{-C}(t)=\det(tI+C)$ denotes the characteristic polynomial
of $-C$. Therefore $p_{-C}(t)=t^m$, which means that $C$ is the zero
matrix.

It follows in particular that $g(x,0)=\det\bigl(\begin{smallmatrix}
E+xA&xB^t\\xB&0\end{smallmatrix}\bigr)$. Since $E$ is invertible, a
Schur complement argument shows that
\begin{equation}\Label{knakpkt}%
g(x,0)\>=\>
x^{2m}\cdot\det(E)\det(-BE^{-1}B^t)+\text{(higher powers of }x).
\end{equation}
We claim that $\det(BE^{-1}B^t)\ne0$. Indeed, if the matrix
$BE^{-1}B^t$ were singular, there would be a vector $w\ne0$ in
$\R^m$ with $B^tw=0$, since $E$ is positive definite.
This would give $\bigl(\begin{smallmatrix}E+xA&xB^t\\xB&0
\end{smallmatrix}\bigr)\binom0w=0$ for every $x\in\R$, implying that
$g(x,0)\equiv0$. Since this case was excluded before, it follows that
$\det(BE^{-1}B^t)\ne0$.

The vanishing order of $g(x,0)$ at $x=0$, which is also the
vanishing order of $P(x,0)=\prod_{i=1}^m(-g_i(x))$, is therefore
equal to~$2m$, by \eqref{knakpkt}.
Since each factor $g_i(x)$ has vanishing order $k_i\ge2$ at $x=0$, we
conclude that $k_i=2$ for $i=1,\dots,m$. In a neighborhood of $O$,
therefore, the zero set of $g(x,y)$ is the union of the $m$ graphs
$y=g_i(x)=a_ix^2+\cdots$ with $a_i\ne0$.
Let the boundary of $K$ near $O$ correspond to $y=g_1(x)$, say. Then
the local positive Nash inequality for $K$ is
$\tilde g(x,y)=y-g_1(x)=y-a_1x^2\cdots$. In particular,
$\partial_x^2\tilde g(0,0)=-2a_1\ne0$. From convexity of $K$ we get
$a_1>0$, and hence $\partial_x^2\tilde g(0,0)<0$, which finally
proves the proposition.
\end{proof}

\begin{lab}\Label{lastep}%
We now finish the proof of Theorem \ref{main2}. Let $f\in\R[x]$ be
hyperbolic with respect to $e$, and assume that the hyperbolicity
cone $C=C_e(f)$ has Nash-smooth boundary.
We have to show that $\sxdeg(C)\le2$, see \ref{spectrahsxdegetc}.
As in \cite[p.~220]{ns} we may assume that $C$ is pointed, and as
there we find a hyperplane $H$ with $e\in H$ for which $S:=C\cap H$
is compact. The restriction $p$ of $f$ to $H$ is real-zero with
respect to~$e$ and satisfies $S_e(p)=S$. The hypothesis of
Nash-smoothness carries over from $\partial C$ to $\partial S$ by
Lemma \ref{isi}. The boundary $\partial S$ cannot contain any
non-trivial line segment, by Nash-smoothness and since $S$ is
compact.
Therefore Proposition \ref{keystep} shows that $\partial S$
has strict positive curvature everywhere. Now Theorem \ref{main1}
implies $\sxdeg(S)\le2$. Since $C$ is the conical hull of $S$, we
finally conclude $\sxdeg(C)\le2$ using \cite[Prop.~1.6]{sch:sxdeg}.
\end{lab}



\begin{thebibliography}{8888}

\bibitem{am}
M.~Artin, B.~Mazur:
On periodic points.
Ann.\ Math.\ (2) \textbf{81}, 82-99 (1965).

\bibitem{av}
G.~Averkov:
Optimal size of linear matrix inequalities in semidefinite approaches
to polynomial optimization.
SIAM J. Appl.\ Algebra Geom.\ \textbf{3}, 128--151 (2019).

\bibitem{BTN}
A.~Ben-Tal, A.~Nemirovski:
\emph{Lectures on Modern Convex Optimization}.
Analysis, Algorithms, and Engineering Applications.
MPS-SIAM Series on Optimization, Philadelphia, PA, 2001.

\bibitem{BCR}
J.~Bochnak, M.~Coste, M.\,F.~Roy:
\emph{Real Algebraic Geometry}.
Erg.\ Math.\ Grenzgeb.\ (3) \textbf{36}, Springer, Berlin, 1998.

\bibitem{fa}
H.~Fawzi:
On representing the positive semidefinite cone using the second-order
cone.
Math.\ Program.\ (Ser.~A) \textbf{175}, 109-118 (2019).

\bibitem{ga}
L.~G\aa rding:
An inequality for hyperbolic polynomials.
J.~Math.\ Mech.\ \textbf{8}, 957--965 (1959).

\bibitem{gu}
O.~G\"uler:
Hyperbolic polynomials and interior point methods for convex
programming.
Math.\ Oper.\ Res.\ \textbf{22}, 350--377 (1997).

\bibitem{hn09}
J.\,W.~Helton, J.~Nie:
Sufficient and necessary conditions for semidefinite representability
of convex hulls and sets.
SIAM J. Optim.\ \textbf{20}, 759-791 (2009).

\bibitem{hn10}
J.\,W.~Helton, J.~Nie:
Semidefinite representation of convex sets.
Math.\ Program.\ (Ser.~A) \textbf{122}, 21-64 (2010).

\bibitem{hv}
J.\,W.~Helton, V.~Vinnikov:
Linear matrix inequality representation of sets.
Comm.\ Pure Appl.\ Math.\ \textbf{60}, 654-674 (2007).

\bibitem{lpr}
A.\,S.~Lewis, P.\,A.~Parrilo, M.\,V.~Ramana:
The Lax conjecture is true.
Proc.\ Am.\ Math.\ Soc.\ \textbf{133}, 2495--2499 (2005).

\bibitem{NP}
T.~Netzer, D.~Plaumann:
\emph{Geometry of Linear Matrix Inequalities}.
Birkh\"auser, Cham, 2023.

\bibitem{ns}
T.~Netzer, R.~Sanyal:
Smooth hyperbolicity cones are spectrahedral shadows.
Math.\ Program.\ (Ser.~B) \textbf{153}, 213-221 (2015).

\bibitem{sch:sxdeg}
C.~Scheiderer:
Second-order cone representation for convex sets in the plane.
SIAM J. Applied Algebra Geometry \textbf{5}, 114-139 (2021).

\bibitem{Sch}
C.~Scheiderer:
\emph{A Course in Real Algebraic Geometry. Positivity and Sums of
Squares}.
Graduate Texts in Mathematics \textbf{303}, Springer Nature, Cham,
2024.

\bibitem{v}
V.~Vinnikov:
LMI representations of convex semialgebraic sets and determinantal
representations of algebraic hypersurfaces: past, present and future.
In: Mathematical methods in systems, optimization and control,
\emph{Oper.\ Theory Adv.\ Appl.} \textbf{222}, 325--349 (2012).

\end{thebibliography}
\end{document}